\documentclass[12pt]{amsart}

\usepackage{amsmath}

\usepackage{amsthm}
\usepackage[dvips]{graphicx}
\usepackage{lscape}
\usepackage[all]{xy}
\usepackage{epsfig}

\usepackage{amssymb}
\usepackage{latexsym}

\newtheorem{Theorem}{Theorem}
\newtheorem{Proposition}{Proposition}
\newtheorem{Lemma}{Lemma}
\newtheorem{Corollary}{Corollary}

\theoremstyle{definition}
\newtheorem{definition}{Definition}
\newtheorem{remark}{Remark}

\def\NN{{\mathcal N}}
\def\CC{{\mathcal C}}
\def\OO{{\mathcal O}}

\def\SO4{{\mathrm{SO(4)}}}
\def\O4{{\mathrm{O(4)}}}

\def\Z{{\mathbb Z}}

\begin{document}

\title{Finite groups acting on hyperelliptic 3-manifolds}


\author[M. Mecchia]{Mattia Mecchia}
\address{M. Mecchia: Dipartimento Di Matematica e Geoscienze, Universit\`{a} degli Studi di Trieste, Via Valerio 12/1, 34127, Trieste, Italy.} \email{mmecchia@units.it}




\date{\today}

\begin{abstract}
We consider 3-manifolds admitting the action of an involution such that its space of orbits   is  homeomorphic to $S^3.$ Such involutions are called \textit{hyperelliptic} as the manifolds admitting such an action. We consider finite groups acting on  3-manifolds and containing  hyperelliptic involutions whose fixed-point set has $r>2$ components. In particular we prove that a simple group containing such an involution is isomorphic to $PSL(2,q)$ for some odd prime power $q$, or to one of  four other small simple groups.
\end{abstract}

\maketitle

\section{Introduction}
In the present paper, we consider \textit{hyperelliptic 3-manifolds}; these are 3-manifolds  which admit the action of a \textit{hyperelliptic involution}, i.e.  an involution with  quotient space homeomorphic to $S^3$. In this paper  3-manifolds are all smooth, closed and orientable, and the actions of groups on 3-manifolds are  smooth and orientation-preserving.

The analogous concept in dimension 2 is a classical research theme and  many papers  about hyperelliptic Riemann surfaces can be found in the literature. In particular the hyperelliptic involution  of  a Riemann surface of genus at least two is unique and central in the automorphism group of the hyperelliptic Riemann surface; as a consequence, the class of automorphism groups of hyperelliptic Riemann surfaces is very restricted (liftings of finite groups acting on $S^2$).

In dimension 3 the presence of a hyperelliptic involution is equivalent to the fact that the 3-manifold is the 2-fold cover of  $S^3$ branched along a link (in brief \textit{the 2-fold branched cover of a link}). 
In fact the fixed-point set of a hyperelliptic involution is a non-empty union of disjoint, simple, closed curves. If $h$ is a hyperelliptic involution acting on $M$, the projection of  the fixed-point set of $h$ to $M/\langle h \rangle\cong S^3$ is  a link and $M$ is the 2-fold branched cover of this link.
On the other hand if $M$ is a 2-fold branched cover of a link, the involution generating the deck transformation group is hyperelliptic. The class of 3-manifolds consisting of  2-fold branched covers of  links has been extensively studied in the literature.

The aim in this paper is to understand which finite  groups can act on a 3-manifold and contain a hyperelliptic involution; in particular we consider the  case of finite simple groups. In contrast with dimension two, hyperelliptic involutions are not unique.

We remark that if $M$ admits the action of a hyperelliptic involution $h$ whose fixed-point set has $r$ components,  then the first $\mathbb{Z}_2$-homology group  of $M$  has rank $r-1$ (see, for example, \cite[Sublemma 15.4]{sakuma-95}). Hence, a hyperelliptic 3-manifold  determines the number of components of the fixed-point set  of its hyperelliptic involutions  and the situation of different numbers of components can be separately analyzed.

The case  of hyperelliptic involutions with connected fixed-point set is particular, since a manifold  admitting such an action is a $\mathbb{Z}_2$-homology 3-sphere. If $G$ is a finite simple group containing  a hyperelliptic involution with connected fixed-point set, then $G$ is isomorphic to the alternating group $\mathbb{A}_5$ \cite{zimmermann2015}. This result is a consequence of a more general one that states that a simple group acting on a $\mathbb{Z}_2$-homology 3-sphere is isomorphic to a linear fractional group $PSL(2,q)$ for some odd prime power  $q$ or to the alternating group $\mathbb{A}_7$ \cite{mecchia-zimmermann2004}. The proof of this result is  highly non-trivial since it is based on the Gorenstein-Walter classification of simple groups with dihedral Sylow 2-subgroups.  We remark that each $\mathbb{Z}_2$-homology sphere is  a rational homology 3-sphere, but for this broader class of 3-manifolds the situation changes completely: each finite group acts on a rational homology 3-sphere (see \cite{cooper-long} for free actions and \cite{boileau-et-al} for non-free actions). A complete classification of the finite groups acting on integer or $\mathbb{Z}_2$-homology 3-sphere is not known and appears to be difficult.

In the present paper  we consider the case of $r>2$ components and  prove the following theorem for simple groups. 

\begin{Theorem}\label{th:simple}
If $G$ is a finite simple group acting  on a 3-manifold $M$ and containing a hyperelliptic involution $h$ such that its fixed-point set  has $r>2$ components then $G$ is isomorphic to a linear fractional group $PSL(2,q),$ for some odd prime power $q,$ or to one of the following four groups: the projective special linear group $\mathrm{PSL}(3, 5),$ the projective general
unitary group $\mathrm{PSU}(3, 3),$ the Mathieu group $ \mathrm{M}_{11}$ and the alternating group  $\mathbb{A}_7.$
\end{Theorem}

We remark that the list of possible groups provided by our theorem is very close to that obtained in the case of $\mathbb{Z}_2$-homology 3-spheres. In the proof of the theorem we use the Gorenstein-Walter classification of simple group with dihedral Sylow 2-subgroups and the Alperin-Brauer-Gorenstein classification of simple group with quasidihedral Sylow 2-subgroups.

It is worth mentioning that   examples of groups containing  a hyperelliptic involution are difficult to construct. The straightforward method consists of taking a finite group of  symmetries of a  link and lifting it to the 2-fold branched cover of the link. In this way we obtain groups with a central hyperelliptic involution. Moreover by the geometrization of 3-manifolds   a finite group acting on $S^3$ is conjugate to a finite subgroup of $\SO4$, so we have a very restricted list of possibilities (see \cite{conway-smith, mecchia-seppi}).

Other examples can be obtained by using Seifert spherical 3-manifolds. 
In particular the  homology Poincar\'e sphere $S^3/I^*_{120},$ the octahedral manifold  $S^3/I^*_{48}$ and  the quaternionic manifold $S^3/Q_8$ admit an action of $\mathbb{A}_5$  that includes a hyperelliptic involution (see \cite{mccullough} for the notations and the computation of the isometry groups and \cite{mecchia-seppi} for a method to compute the space of orbits of the involutions).  The numbers of components of the fixed-point set of the hyperelliptic involutions in the examples are 1, 2 and 3, respectively.

At the moment $\mathbb{A}_5$ is the only known simple group  admitting an action  on a 3-manifold and containing a hyperelliptic involution.  

\vskip1em

A key ingredient in the proof of Theorem~\ref{th:simple} is the analysis of the Sylow 2-subgroups; in particular we prove the following proposition which is interesting in its own right.

\begin{Proposition}\label{prop:generated}
Let $G$ be a finite 2-group acting on a 3-manifold and containing  a  hyperelliptic involution whose fixed-point set has $r>2$ components. The subgroup  generated by all hyperelliptic involutions in $G$  has order two or is dihedral. 
\end{Proposition}

Since  by Thurston's Orbifold  Geometrization theorem  each hyperelliptic involution of a hyperbolic 3-manifold is conjugate to an isometry  and the isometry groups of  hyperbolic 3-manifolds are finite,  we can easily derive from Proposition~\ref{prop:generated} the following corollary.

\begin{Corollary} \label{cor:bound-for-knots}
There are at most three inequivalent links with at least 3 components which have  the same hyperbolic 2-fold branched cover.
\end{Corollary}

This upper bound  was already claimed in \cite{mecchia-zimmermann2004-bis} where a proof  based on a result contained in \cite{reni-zimmermann2001} was given; however,  in the proof in \cite{reni-zimmermann2001} we found a gap that seems not easy to fill. In the present paper the proof of Corollary \ref{cor:bound-for-knots}  is independent of the result in   \cite{reni-zimmermann2001}. We remark that  in \cite{mecchia-zimmermann2004-bis} three is proved to be the optimal bound.

In this paper also the case of arbitrary groups is considered.  If the   hyperelliptic involution  has a fixed-point set  with odd number of components  our methods can be extended: in Section~\ref{sect:odd-components} we prove a theorem which describes the structure of an arbitrary    group acting on a 3-manifold and containing a hyperelliptic involution whose fixed-point set has $r$ components with $r>2$ and odd.

We remark that   the case $r=2$ remains completely open. Considering our approach, one of the main ingredients is the analysis of the centralizers of the hyperelliptic involutions but, if the components of the fixed-point set are two,  this analysis is much more difficult.

From now on $G$ is a finite group acting smoothly and orientation preservingly on a closed, orientable, smooth 3-manifold $M.$

\subsection*{Organization of the paper} In Section~\ref{sect:centralizer}, we consider  centralizers of hyperelliptic involutions. In Section~\ref{sect:2-groups}, we prove several properties of finite 2-groups containing a hyperelliptic involution whose fixed-point set has $r>2$ components.   In Section~\ref{sect:simple-groups} we prove Theorem~\ref{th:simple}. Finally, in Section \ref{sect:odd-components} we consider arbitrary groups containing a hyperelliptic involution whose fixed-point set has $r$ components with $r>2$ and $r$  odd.

\vskip1em

\section{Centralizers of hyperelliptic involutions}\label{sect:centralizer}

In this section we collect  some properties about the elements of $G$ commuting with a hyperelliptic involution,  in particular the centralizer of a hyperelliptic involution is of very special type.

\begin{remark}\label{rem:stabilizer-simple-closed-curve}

Let $K$ be a simple closed curve  in $M$ and let   $I$ be the subgroup of $G$ fixing $K$ setwise.
Since $G$ is finite,  we can choose a $G$-invariant Riemannian metric on $M$ with respect to which $G$ acts by isometries.
An  element $f$ of $I$ can act on $K$ either as rotation  or as a reflection.
In the first case we call $f$ a $K$-rotation, while in the second case  a $K$-reflection.
The elements acting trivially on $K$ are considered $K$-rotations.
The subgroup of $K$-rotations is abelian of rank at most two and its index in $I$ is at most two.
The  conjugate  of  a $K$-rotation by a $K$-reflection is the  inverse of the $K$-rotation. 
We can conclude that  $I$ is isomorphic to a subgroup of a semidirect product
$\mathbb{Z}_2 \ltimes (\mathbb{Z}_n\times\mathbb{Z}_m)$, for some
non-negative integers $n$ and $m$, where $\mathbb{Z}_2$ operates on the normal
subgroup $\mathbb{Z}_n\times\mathbb{Z}_m$
by sending each element to its inverse.

\end{remark}

\begin{Proposition}\label{prop:centralizer} If $h\in G$ is a hyperelliptic involution then the quotient $C_G(h)/\langle h \rangle$ is isomorphic to a finite subgroup of $\SO4.$

\end{Proposition}

\begin{proof} The group $C_G(h)/\langle h \rangle$ acts on  $M/\langle h \rangle$ that is homeomorphic to $S^3$.  The Thurston orbifold geometrization theorem (see~\cite{boileau-leeb-porti}) and  the spherical space form  conjecture for free actions on $S^3$  proved by Perelman (see~\cite{morgan-tian}) imply that every finite group of diffeomorphisms of the 3-sphere is conjugate to a finite subgroup of $\textrm{SO}(4)$.
\end{proof}

The two following technical lemmas concerning elements commuting with hyperelliptic involutions are used several times in the paper.

\begin{Lemma}\label{lem:two-commuting-hyperelliptic}
Let $h_1, h_2$ be two commuting hyperelliptic involutions of $G$ and $E=\langle h_1, h_2\rangle$. Then  $\CC_G(E)/\langle h_1 \rangle$ is isomorphic to a subgroup of $\mathbb{Z}_2 \ltimes (\mathbb{Z}_n\times\mathbb{Z}_m)$, for some
non-negative integers $n$ and $m$, where $\mathbb{Z}_2$ operates on the normal
subgroup $\mathbb{Z}_n\times\mathbb{Z}_m$ by sending each element to its inverse.

\end{Lemma}

\proof
The group $\CC_G(E)/\langle h_1 \rangle$ acts on the $3$-sphere $M/\langle h_1 \rangle$ and contains in its center  $h_2\langle h_1 \rangle.$ Since $h_2$ has non-empty fixed-point set on $M$,  the fixed-point set of  $h_2\langle h_1 \rangle$ on $M/\langle h_1 \rangle$  is a simple closed curve fixed by $C_G(E)/\langle h_1 \rangle$ setwise. The claim follows from Remark~\ref{rem:stabilizer-simple-closed-curve}.\qed

\begin{Lemma}\label{lem:r-components}\cite[Lemma 1]{reni-zimmermann2001}
Let $h$ be a hyperelliptic involution in $G$, 
and suppose that the fixed-point set  of $h$ has $r>2$
components. Let $x$ be an element of $G$  different
from $h$ that commutes with $h$.
Then the fixed-point set of $x$ has also $r$ components if and only if $x$
is a strong inversion of  the fixed-point set of $h$ (that is
$x$ acts as a reflection on each component of the fixed-point set of $h$).
\end{Lemma}

\section{Finite 2-groups containing hyperelliptic involutions}\label{sect:2-groups}

In this section we analyse geometric and algebraic properties of 2-groups acting on a 3-manifold and containing a hyperelliptic involution whose fixed point set has $r>2$ components. If $g$ is an element of $G$, we denote by $Fix(g)$ the fixed-point set of $g.$

\begin{Lemma}\label{lem:elementary-abelian-group}
Suppose that $G$ is an elementary abelian $2$-group. If $G$ contains a hyperelliptic involution $h$ such that $Fix(h)$  has $r>2$ components, then $G$ contains at most three hyperelliptic involutions which generate a subgroup of rank at most two. If the group generated by the hyperelliptic involutions has rank two, then it contains all the elements of $G$ whose fixed-point set has $r$-components.
\end{Lemma}

\proof
We denote by $K$ one of the components of $Fix(h)$.
We can suppose that $G$ is generated by involutions  whose fixed-point set has $r$-components. By Lemma~\ref{lem:r-components} each element in $G$ leaves $K$ invariant. Remark~\ref{rem:stabilizer-simple-closed-curve} implies that $G$ has rank at most three. If $G$ has rank one or two the thesis trivially holds, thus we can suppose that $G$ has rank three.

Since  $G$ has rank three, $G$ contains also two involutions different from $h$  acting on $K$ as rotations; if we denote one of these rotations by $y$, the other one is $hy$. By  Lemma~\ref{lem:r-components}, the fixed-point set of $y$ and $hy$ can not have $r$ components. 
If $h$  is the only hyperelliptic involution in $G$ the thesis trivially holds, thus we can assume the existence of $t$, a hyperelliptic involution different from $h.$ By  Lemma~\ref{lem:r-components}   hyperelliptic involutions  in $G$ distinct from $h$  are $K$-reflections. 
The $K$-reflections in $G$ are $t$, $th$, $ty$ and $thy$. If the fixed-point set of $ty$ had $r$-components (in particular if it was hyperelliptic), then $ty$ and $y$  should be  strong inversions of $Fix(t)$ but this is impossible since $Fix(y)$ does not have $r$ components. The same holds for $thy.$ This implies that all the involutions whose fixed-point set has $r$ components are contained in the group generated by $t$ and $h$; the hyperelliptic involutions  are at most three.
\qed

\begin{Lemma}\label{lem:central-element}
Suppose that $G$ is a 2-group containing a hyperelliptic involution $h$ such that $Fix(h)$  has $r>2$ components, then there exists $x$ an involution (not necessarily hyperelliptic) such that $G=C_G(x)$ and $Fix(x)$ has $r$ components.
\end{Lemma}

\proof

This Lemma is implied by the proof of Theorem 1 in \cite{reni-zimmermann2001}. We give here the proof for the sake of completeness.


Suppose that $h$ is not central in $G$ and define $N$ the subgroup of $C_G(h)$ generated by all involutions whose fixed-point set has exactly $r$ components. By Lemma~\ref{lem:r-components}, the non-trivial elements of $N$ act as reflections on $Fix(h)$, thus  $N$ fixes setwise all the components of $Fix(h)$ and Remark~\ref{rem:stabilizer-simple-closed-curve} applies to $N$.  Since $C_G(h)$ is a proper subgroup of  $G$, by \cite[Theorem 1.6.]{suzuki1} there exists an element $f$ of $G\setminus C_G(h)$ normalizing $C_G(h)$. The element $t:=fhf^{-1}$  is a hyperelliptic involution commuting with $h$, thus $t$ is in $N$. 
Moreover $t$ is central in $C_G(h)$ and by Remark~\ref{rem:stabilizer-simple-closed-curve}, the subgroup  $N$ is  elementary abelian of  rank two or three. 
 

By Lemma~\ref{lem:elementary-abelian-group} we obtain that $N\cong\mathbb{Z}_2 \times \mathbb{Z}_2$ and we will show that $G=C_G(th)$. We remark that $f$ normalizes $N$ and, since its order is a power of two, it centralizes  $th$. We obtain that $C_G(th)$ contains properly $C_G(t)=C_G(h).$

Denote by $g$ an element normalizing $C_G(th)$; the element $gthg^{-1}$ is in the center of $C_G(th)$, so it commutes with $t$ and $h$. The fixed point set of $gthg^{-1}$ has $r$ components, thus $gthg^{-1}$ is contained in $N$. If  $gthg^{-1}=h$ or $gthg^{-1}=t$, we obtain $C_G(th)=C_G(t)=C_G(h)$ getting a contradiction. The only  possibility is $gthg^{-1}=th$. By  \cite[Theorem 1.6.]{suzuki1} we obtain $G=C_G(th).$ The fixed-point set of $th$ has $r$ components, for $th$ is a strong inversion  of $Fix(h)$; this concludes the proof of the Lemma. 

\qed

We remark that in this paper we include between the dihedral groups the elementary abelian group of rank 2.

\begin{Proposition}\label{prop:generate-dihedral}
Suppose that $G$ is a 2-group containing a hyperelliptic involution $h$ such that $Fix(h)$  has $r>2$ components.   If $H$ is the subgroup of $G$ generated by the hyperelliptic involutions in $G$, then $G$ either has order two or is dihedral. 
\end{Proposition}

\proof

By  Lemma~\ref{lem:central-element}, the group $G$ contains a central involution $x$ such that $Fix(x)$ has $r$-components.  By Lemma~\ref{lem:r-components} a hyperelliptic involution in $G$ either coincides with $x$ or is a strong inversion of $Fix(x)$, so any hyperelliptic involution leaves invariant each component of $Fix(x)$. Let $K$ be one of the components of $Fix(x)$. A hyperelliptic  involution that is different from $x$ is a  $K$-reflection and each element in $H$ leaves $K$ setwise invariant. By Remark~\ref{rem:stabilizer-simple-closed-curve} we obtain  that $H$ is a subgroup of a semidirect product $\mathbb{Z}_2 \ltimes (\mathbb{Z}_{2^a}\times\mathbb{Z}_{2^b})$,  where $a$ and $b$ are positive integers. 
 
By contradiction, we suppose that $H$ is  neither cyclic nor dihedral.
Since $H$ is not cyclic, it contains $t$  a $K$-reflection that is hyperelliptic. We obtain that $H\cong\mathbb{Z}_2 \ltimes (\mathbb{Z}_{2^a}\times\mathbb{Z}_{2^b})$, with $a\geq 1$ and $b \geq 1$ and $x$ is contained in $H$. We denote by $g$ and $f$ two generators of the subgroup of $K$-rotations, respectively of order $2^a$ and $2^b.$ We recall that  in $H$ we have four conjugacy classes of $K$-reflections: $\{tf^kg^h\,|\, k \text{ odd and } h \text{ even}\},$ $\{tf^kg^h\,|\, k \text{ even and } h \text{ odd}\},$ $\{tf^kg^h\,|\, k \text{ even and } h \text{ even}\}$ and $\{tf^kg^h\,|\, k \text{ odd and } h \text{ odd}\}.$

If $a=b=1$ then $H$ is an elementary abelian group of rank three and Lemma~\ref{lem:elementary-abelian-group} gives a contradiction. 


Then we suppose that  $a>1$ and $b>1$.  
 The elements $t,$ $tg^{2^{a-1}}$ and $tf^{2^{b-1}}$   are all hyperelliptic, for they are conjugate in $H$, and generate an elementary abelian subgroup of rank three. This gives a contradiction by Lemma~\ref{lem:elementary-abelian-group}.

Finally we assume that $a=1$ and $b>1$. Since the elements $t$ and  $tf^{2^{b-1}}$ are conjugate, also $tf^{2^{b-1}}$ is  hyperelliptic. Consider the group generated by $t,$ $f^{2^{b-1}}$ and $g$, which is an elementary abelian group of rank three. By Lemma~\ref{lem:elementary-abelian-group} the only other involution in this group, that can be hyperelliptic, is $f^{2^{b-1}}.$ Thus the elements of type $tf^{k}g$ with $k$ even can not be hyperelliptic. Consider now the group generated by $tf,$ $f^{2^{b-1}}$ and $g$. If $tf$ is hyperelliptic, then so is $tf^{2^{b-1}+1}$.   By Lemma ~\ref{lem:elementary-abelian-group} the involutions in the conjugacy class of $tfg$ can not be hyperelliptic. In this case the group generated by $t$ and $f$ is dihedral and contains all the hyperelliptic involutions. If $tfg$ is hyperelliptic, we can reverse the roles and we obtain that all the hyperelliptic rotations are contained in the dihedral subgroup generated by $t$ and $fg.$

\qed



\begin{Lemma}\label{lem:quotient}
Suppose that $N$ is a 2-subgroup of $G$  generated by hyperelliptic involutions, then the quotient  $M/N$ is homeomorphic to $S^3$.
\end{Lemma}
\proof
By Proposition~\ref{prop:generate-dihedral} the group $N$ is either cyclic of order two or dihedral. In the first case the thesis trivially holds. In the second case $N$ is generated by  two hyperelliptic involutions that we denote by  $t$ and $s$. Let $2^a$ be the order of $ts$ (i.e. $2^{a+1}$ is the order of $N$). We consider the following subnormal series of subgroups of $N$: 
$$N_1=\langle t \rangle \subset N_2=\langle N_1, t(ts)^{2^{a-1}} \rangle \subset \langle N_2, t(ts)^{2^{a-2}} \rangle\subset\ldots N_a=\langle N_{a-1}, t(ts)^{2} \rangle\subset N=\langle N_{a}, s \rangle $$
All the involutions in $N$, but one that is central in $N$, are conjugate either to $t$ or to $s$, thus all the involutions, but possibly one, are hyperelliptic.    
Since $t$ is hyperelliptic $M/N_1$ is homeomorphic to $S^3.$ We recall that, by the positive solution of the Smith Conjecture, if an involution acts with non-empty fixed-point set on $S^3$, then its action is linear and the underlying topological space of the quotient orbifold of $S^3$ by this involution  is again a 3-sphere.   Since $N_1$ is normal in $N_2$, the quotient group $N_2/N_1$ induces an action on $M/N_1$. In particular $t(ts)^{2^{a-1}}$ projects to an involution acting faithfully on $M/N_1$ and generating $N_2/N_1$. As the fixed-point set of $t(ts)^{2^{a-1}}$ is non-empty, so is the fixed-point set of the projected involution. We can obtain  $M/N_2$ as a double quotient  $(M/N_1)/(N_2/N_1)$ and, since $N_2/N_1$ is generated by an involution not acting freely, $M/N_2$ is again homeomorphic to $S^3$. 
We can iteratively apply  this argument to each group $N_i$ of the subnormal series, and finally we get the thesis.

\qed

We recall that the 2-rank of a finite group is the maximum of the ranks of  elementary abelian 2-subgroups. 
The following proposition assures us that the 2-rank of a finite group containing a hyperelliptic involution is small.

\begin{Proposition}\label{prop:2-rank}
If  $G$  contains  a hyperelliptic involution $h$ such that $Fix(h)$  has $r>2$ components, then the 2-rank of $G$ is at most four.

\end{Proposition}

\proof

We denote by $S$ a Sylow 2-subgroup of $G$ and by $E$ one elementary abelian subgroup of maximal order.  Since the conjugate of a hyperelliptic involution is hyperelliptic, we can suppose that $S$ contains $h$ by the Sylow Theorems. Let $H$ be the normal  subgroup of $S$ generated by all the hyperelliptic involutions in $S$. By Proposition~\ref{prop:generate-dihedral} we obtain that $H$ is cyclic of order two or dihedral. The subgroup  $E$ projects to an elementary abelian group acting on $M/H$, which is isomorphic to $EH/H\cong E/E\cap H.$  Since by Lemma~\ref{lem:quotient} the quotient $M/H$ is homeomorphic to $S^3$, the group $E/E\cap H$ admits a faithful  action on $S^3$ and hence  has rank at most three (see \cite[Propositions 2 and 3]{mecchia-zimmermann2004}.) 
The group $E\cap H$ has rank at most two. If $E\cap H$ is cyclic, we are done. If  $E\cap H\cong \mathbb{Z}_2\times\mathbb{Z}_2$ then $E$ contains a hyperelliptic involution $t$ (at most one involution in $H$ can be non-hyperelliptic).  By using the action induced by $E$ to $M/\langle t \rangle$, which is homeomorphic to $S^3$, we can conclude also in this case that the rank of $E$ is at most four.

\qed

\section{Finite simple groups containing hyperelliptic involutions}\label{sect:simple-groups}

In this section we prove Theorem~\ref{th:simple}. We consider first the following algebraic definitions:

\begin{definition}\label{def:closed-subgroups}
Let $F$ be a finite group with two subgroups $T$ and $C$ such that  $ C\leq  T \leq F$.
\begin{enumerate}
\item  $C$ is said to be \textit{weakly closed} in $T$, if $ C^f \leq T$ implies $ C^f=C$, for each $f \in F$.

\item $C$ is said to be \textit{strongly closed} in $T$, if $ C^f \cap T  \leq C$, for each $f \in F$.

\item Let $I$ be the set of the involutions of $C$; the subgroup $C$ is said to be \textit{strongly involution closed} in $T$, if $I^f \cap T \subset C$, for each $f\in F$.

\end{enumerate}

\end{definition}

If $S$ is a Sylow 2-subgroup of $G$, the subgroup of $S$ generated by hyperelliptic involution is weakly closed in $S$ and   is cyclic or  dihedral by Proposition~\ref{prop:generate-dihedral}.

The groups containing a dihedral 2-subgroup weakly closed in the Sylow 2-subgroup were studied by Fukushima in \cite{fukushima}. We will use his result to strongly restrict the possible simple groups containing  a hyperelliptic involution with fixed-point set with $r>2$ components.

The other two definitions (strongly closed and strongly involution closed subgroups) will be used in the next session where the  case of arbitrary groups will be considered.

In the proof of Theorem~\ref{th:simple} we  will find simple groups with Sylow 2-subgroups that are  dihedral, semidihedral or  wreath products  $\Z_{2^m} \wr \Z_2.$

We recall that  a semidihedral group of order $2^n$ with $n\geq 4$ has the following presentation:
$$\langle s,\,f\,|\, s^2=f^{2^{n-1}}=1,\, sfs=f^{2^{n-2}-1}\rangle.$$    

We remark that in literature these groups  are also called quasidihedral. For the basic properties of semidihedral groups see \cite[Lemma 1]{alperin-brauer-gorenstein}.

The wreath product $\Z_{2^m} \wr \Z_2$ is a semidirect product  $(\Z_{2^m} \times  \Z_{2^m})\rtimes \Z_2$ where $\Z_2$ acts on a given couple of generators of $\Z_{2^m} \times  \Z_{2^m}$ exchanging them.

\vskip1em

\textit{Proof of Theorem~\ref{th:simple}} Let $S$ be a Sylow 2-subgroup of $G$ containing $h$ and let $D$ be the subgroup of $S$ generated by all the hyperelliptic involutions of $S$. If $h$ was the only hyperelliptic involution in $S$, then by Glauberman's $Z^*-$ theorem $h$ would be in the center of $G$, and this is impossible since $G$ is simple. So we can suppose that $G$ contains at least two hyperelliptic involutions.  Then, $D$ is weakly closed and dihedral by Proposition~\ref{prop:generate-dihedral}. By~\cite[Theorem~1]{fukushima}, since $G$ is simple, we get that either $G$ is isomorphic to $PSU(3,4)$, or $S$ is dihedral, semidihedral or a wreath product  $\Z_{2^m} \wr \Z_2$. 

 If $G\cong PSU(3,4)$ there is a unique class of involutions, that should be all hyperelliptic, and  $S$ should contain an elementary abelian subgroup of rank 4 (see~\cite{conway-et-al}); this contradicts Proposition~\ref{prop:generate-dihedral}.

When $S$ is dihedral  we can conclude by using the Gorenstein-Walter classification of  finite simple groups with dihedral Sylow  $2$-subgroups \cite{gorenstein-walter}.

If  $S$ is quasidihedral  we can use the result of Alperin, Brauer and Gorenstein~\cite{alperin-brauer-gorenstein} and we obtain that $G$ is isomorphic to $PSL(3, q)$ for some odd prime power $q$ such that $q\equiv -1 \bmod 4$ or  to $PSU(3, q)$ for some odd prime power  $q$ such that $q\equiv 1 \bmod 4$ or  to $M_{11}$.
To exclude most part of these groups we can use the structure of the centralizer of the involutions which for groups of Lie type is described~\cite[Table 4.5.2.]{gorenstein-lyons-solomon3}. We remark that in these groups there exists a unique conjugacy class of involutions and so each involution is hyperelliptic. By Proposition~\ref{prop:centralizer} and the classification of the finite 
subgroups of $\textrm{SO}(4)$ (see~\cite{conway-smith} or \cite{mecchia-seppi}) we know that a non-abelian  simple section of any centralizer of an involution is isomorphic to $\mathbb{A}_5.$
Analyzing \cite[Table 4.5.2.]{gorenstein-lyons-solomon3} we can conclude that the only possible groups are  $PSL(3, 5),\,PSU(3, 3)$ and $M_{11}.$

Suppose finally that $S$ is a wreath product $\Z_{2^m} \wr \Z_2$. Let $s,t\in S$ be elements of order $2^m$ and $2$, respectively, such that $S_0:=\langle s, s^t\rangle\cong  \Z_{2^m}\times \Z_{2^m}$ and $S=S_0\langle t\rangle$. Note that an element $\mu=s_0t\in S\setminus S_0$ ha order $2$ if and only if $s_0^t=s_0^{-1}$, that is $s_0\in \langle s^{-1}s^t\rangle$. Thus, the involutions in $S\setminus S_0$ are exactly $2^m$ and $S$ contains $2^m+3$ involutions. Since $G$ is simple, by Corollary 2G, Proposition $2\mathrm{F}_a$ and Proposition $2\mathrm{B}_\mathrm{I}$ in~\cite{brauer-wong}, all the involutions of $S$ are conjugate in $G$. Hence they are all hyperelliptic, and they generate a dihedral subgroup by Proposition~\ref{prop:generate-dihedral}, of order at most $2^{m+1}$. It follows that the involutions in $S$ are at most $2^m+1$, a contradiction. \qed

\section{The case of hyperelliptic involutions with fixed-point set with  odd number of components}\label{sect:odd-components}

If we drop  the hypothesis that $G$ is  simple,  the conditions  given by Fukushima's result on the structure of $G$ are too weak to lead us to an effective restriction of the possible groups. On the other hand if we suppose that the number of components of the fixed-point set of  hyperelliptic involutions is odd, we are able to prove the existence of a dihedral 2-subgroup strongly involution  closed and we can use the stronger results in \cite{goldschmidt} and \cite{hall}.

\subsection{Existence of a strongly involution closed dihedral subgroup}

We begin with an algebraic remark.

\begin{remark}\label{rem:normalizer-coprime-quotient} Let $F$ be a finite group. Suppose that $N$ is a normal subgroup of $F$ and $P$ is a $p$-subgroup of $F$. If the order of $N$ is coprime to $p$, then the normalizer of the projection of $P$ to  $F/N$ is the projection of the normalizer of $P$ in $F$, that is $$\NN_{F/N}(PN/N)=(\NN_F(P)N)/N.$$ The inclusion $\supseteq$ holds trivially. We prove briefly the other inclusion. Let $fN$ be an element of $\NN_{F/N}(PN/N)$, then $P^f\subseteq PN$. Both $P^f$ and $P$ are Sylow  $p$-subgroups of $PN$ and by the second Sylow Theorem they are conjugate by an element $tn\in PN.$   We obtain that $P^{ftn}=P$, and hence $f\in \NN_F(P)N.$ By the same  argument we can prove  that two p-groups $P$ and $U$ are conjugate in $F$ if and only if $PN/N$ and $UN/N$ are conjugate in $F/N.$ Moreover the analogous  following formula  holds for  centralizers: $$\CC_{F/N}(PN/N)=\CC_F(P)N/N.$$ The inclusion $\supseteq$ holds trivially again. If $fN\in \CC_{F/N}(PN/N)$ then $fN$ normalizes $(PN)/N$ and we can suppose by the first part of the remark that $f\in \NN_F(P)$. Moreover  for each $t\in P$ there exists $n\in N$ such that   $t^f=tn$.  Since $t^f\in P$ and $P\cap N$ is trivial, $n=1$ for each $n$ and $f\in \CC_F(P).$   

\end{remark}


\begin{Proposition}\label{prop:s-i-closed-in-odd-case}
Let $S$ be a Sylow 2-subgroup of $G$ containing a hyperelliptic involution $h$ such that $Fix(h)$ has $r>2$ components. If $r$ is an odd number, then either $h$ is the unique hyperelliptic involution contained in $S$ or the subgroup $B$ generated by all the involutions of $S$ whose fixed-point sets have $r$ components  is dihedral and strongly involution closed in $S$. 
\end{Proposition}

\begin{proof}
 We may assume that there are at least two hyperelliptic involutions in $G$; we denote by $D$ the subgroup of $S$ generated by the hyperelliptic involutions. The group $D$ is dihedral by Proposition~\ref{prop:generate-dihedral} and $D\leq B$. We denote by $2^{n+1}$ the order of $D.$ If $B=D$ then the fixed-point set of each involution in $B$ has $r$-components and we are done. So assume that $B>D$.  Write $D=\langle f, h\rangle$, where $f$ generates a cyclic subgroup of order $2^n$ and $h$ is a hyperelliptic involution inverting $f$. We denote by $L$ the fixed-point set of $f^{2^{n-1}}.$ We remark that  $f^{2^{n-1}}$ can be hyperelliptic or not, but since $f^{2^{n-1}}$ can be seen as a product of two commuting hyperelliptic involutions,  in any case the fixed-point set $L$ has $r$ components by Lemma~\ref{lem:r-components}. 
\medskip

\textbf{Claim:}\textit{For every involution $t\in B\setminus D$ whose fixed-point set has $r$ components, $t$ acts as a strong inversion on each component of $L$ and $\langle t, D\rangle$ is dihedral.} 

By Lemma~\ref{lem:r-components} all the elements of $D$ fix setwise each component of $L.$ If $n>1$ the element  $t$ centralizes $f^{2^{n-1}};$ if $n=1$ the element $t$ centralizes at least one involution in $D$ and we can suppose that $f^{2^{n-1}}$ commutes with $t.$  Since $t$ has order two and $r$ is odd, $t$ fixes setwise one of the components of $L$, which we denote by $K.$  If $t$ acted as a rotation on $K$ it would commute with the hyperelliptic involution; in particular $f^{2^{n-1}},$ $t$ and $h$ would generate an elementary group or rank three contradicting Lemma~\ref{lem:elementary-abelian-group}, so $t$ acts on $K$ as a reflection and  each element of $\langle t, D\rangle$ fixes setwise the simple closed curve $K$. By Remark~\ref{rem:stabilizer-simple-closed-curve} the group $\langle t, D\rangle$ is either dihedral or isomorphic to $\mathbb{Z}_2\times \mathbb{D}_{2^n}.$ 
 If $\langle t, D\rangle$ was not  dihedral, $t$ would commute with a hyperelliptic involution different from  $f^{2^{n-1}}$ and $\langle t, D\rangle$ contains an elementary subgroup of rank three contradicting Lemma~\ref{lem:elementary-abelian-group}.

It remains to prove that $t$  acts as a strong inversion on each component of $L.$
The involution $t$ projects to $\bar t$, an involution acting on  the quotient $M/D$. Since by Lemma~\ref{lem:quotient} $M/D$ is homeomorphic to $S^3$ , the involution  $\bar t$ has non-empty and connected fixed-point set. Since $\langle t, D\rangle$ is dihedral, the only non trivial element in $D$ leaving invariant $Fix(t)$ is $f^{2^{n-1}}.$ Since $r$ is odd $f^{2^{n-1}}$ leaves invariant at least a component of $Fix(t)$, then the only possibility to get a connected fixed-point set for $\bar t$ is  that the $f^{2^{n-1}}$ acts as a strong inversion on  each component of $Fix(t).$
Since each component of $L$ can intersect $Fix(\bar t)$ in at most two points, we obtain that $t$ acts as a strong inversion on each component of $L.$ This concludes the proof of the Claim.

\medskip
By the previous Claim and Lemma~\ref{lem:r-components}, all the generating involutions of $B$ leave invariant all the components of $L$, in particular $K$; therefore $B$ is isomorphic to a subgroup of  a semidirect product $\mathbb{Z}_2 \ltimes (\mathbb{Z}_{2^a}\times\mathbb{Z}_{2^b})$   where $\mathbb{Z}_2$ operates on the normal
subgroup $\mathbb{Z}_{2^a}\times\mathbb{Z}_{2^b}$ by sending each element to its inverse (see Remark~\ref{rem:stabilizer-simple-closed-curve}).
Assume by contradiction that  $B$ is not dihedral. By Lemma~\ref{lem:r-components}, the subgroup $B$ is not elementary abelian. Then the center of $B$ contains an involution $c$ not contained in $D$ and $c$ centralizes an element of order at least $4$. It follows that $c$ acts as a non-trivial rotation on every component of $K$  and hence $ch$  acts as a strong inversion on $K$. By Lemma~\ref{lem:elementary-abelian-group} we get that $ch$ has $r$ components and the Claim yields that $\langle D, ch\rangle=D\times \langle c\rangle$ is dihedral: a contradiction. Since $B$ is dihedral, the fixed-point set of each involution in $B$ has $r$  components and $B$ is strongly involution closed. 
\end{proof}

\subsection{Structure theorem for arbitrary groups}\label{possible-groups}

We recall some algebraic definitions we use in this subsection. 
A finite group $Q$ is \emph{quasisimple} if it is perfect (the
abelianised group is trivial) and the factor group $Q/\mathcal{Z}(Q)$ of $Q$ by its centre
$\mathcal{Z}(Q)$ is a non-abelian simple group (see \cite[chapter 6.6]{suzuki2}). A group $E$ is 
\emph{semisimple} if it is perfect and the factor group $E/\mathcal{Z}(E)$ is a direct 
product of non-abelian simple groups. A semisimple group $E$ is a central 
product of quasisimple groups which are uniquely determined. Any finite group 
$F$ has a unique maximal semisimple normal subgroup $\mathcal{E}(F)$ (maybe trivial), 
which is characteristic in $F$. The subgroup $\mathcal{E}(F)$ is called the \emph{layer} 
of $F$ and the quasisimple factors of $\mathcal{E}(F)$ are called the \emph{components} 
of $F$.

The maximal normal nilpotent subgroup of a finite group $F$ is called the 
\emph{Fitting subgroup} and we denote it  by $\mathcal{F}(F)$. The Fitting subgroup 
commutes elementwise with the layer of $F$. The normal subgroup generated by 
$\mathcal{E}(F)$ and by $\mathcal{F}(F)$ is called the \emph{generalised Fitting subgroup} and we denote it by $\mathcal{F}^*(F)$. The generalised Fitting subgroup has the 
important property to contain its centraliser in $F$, which thus coincides 
with the centre of $\mathcal{F}^*(F)$. For further properties of the generalised Fitting 
subgroup see \cite[Section 6.6.]{suzuki2}. 

If $F$ is a finite group we denote by $\OO(F)$ the maximal normal subgroup of odd order of $F.$

\begin{Proposition}\label{prop:component-normally-generated} 
 Let  $S$ be a Sylow 2-subgroup of $G$ containing a hyperelliptic involution $h$ such that $Fix(h)$ has $r>2$ components and let $B$ be  the subgroup of $S$ generated by the involutions whose fixed-point sets have $r$ components and let $N$ be the subgroup normally generated by $B.$
If $r$ is an odd number, then one of the following conditions holds:

\begin{enumerate}
\item $S$ contains a unique hyperelliptic involution, the subgroup $\langle h,\OO(G) \rangle/\OO(G)$ lies in the centre of $G/\OO(G)$ and  $G/\langle h,\OO(G) \rangle$ is isomorphic to a  quotient of a finite subgroup of $\SO4$;

\item  $S$ contains at least two hyperelliptic involutions, the Sylow $2$-subgroup of $N$  is dihedral  and  one of these two situations occurs:

\begin{enumerate}
 \item  $N=B \OO(N)$
\item $ N/\OO(N)$ is isomorphic to $\mathbb{A}_7$, $\textrm{PGL}(2,q)$ for an odd prime power $q$, or  $\textrm{PSL}(2,q)$ for an odd prime power  $q>5$.

\end{enumerate}

\end{enumerate}

\end{Proposition}

\begin{proof}
If $h$ is the only hyperelliptic involution in $S$, then we can use the Glauberman's $Z^*-$ theorem and obtain that $h\OO$ is contained in the center of $G/\OO.$ By  Remark~\ref{rem:normalizer-coprime-quotient} and Proposition~\ref{prop:centralizer}, we get that $(1)$ holds. We  assume that $S$ contains at least two hyperelliptic involutions. 
By Proposition~\ref{prop:s-i-closed-in-odd-case}, the subgroup $B$  is  a strongly involution closed dihedral subgroup of $G.$ By~\cite[Theorem 3.1]{hall} we obtain that either $B$ is strongly closed in $S$ or $B$ is contained with index $2$ in a strongly closed semidihedral subgroup $Q$. 

We show that the latter case does not occur. 
Suppose by contradiction  that $B$ is  a maximal subgroup in a strongly closed semidihedral subgroup $Q$ and denote by $c$ the unique central involution of $Q$. Then $Fix(c)$ has $r$ components by Lemma~\ref{lem:central-element} and there exists at least another hyperelliptic involution that is not central in $Q$.
All the non central involutions of $Q$ are conjugate (see \cite[Lemma 1]{alperin-brauer-gorenstein}); so every non central involution of $Q$ is hyperelliptic and,  by Lemma~\ref{lem:r-components}, acts as a reflection on each component of $Fix(c)$. Let $f$ be an element of order $4$ in $Q$ not generating a normal subgroup of $Q.$ We note that $f$ leaves invariant $Fix(c).$ 
Since $Fix(c)$ has an odd number of components, $f$ fixes setwise one component of  $Fix(c)$. By Remark~\ref{rem:stabilizer-simple-closed-curve} the element $f$   acts  as a rotation on the fixed component; this implies that $f$ and  any non central involution in $Q$ generate an abelian or dihedral group and this is impossible.  

Therefore we have that $B$  is strongly closed in $S$ and by the main theorem of~\cite{hall} cases $(2a)$ or $(2b)$ hold, or $N/\OO(N)\cong PSU(3,4)$. In this case, it is straightforward to see that $D$ is an elementary abelian group of order $4$, the subgroup $D$ coincides with the center of a Sylow $2$-subgroup $T$ of $N$ and $T/Z(T)$ is elementary abelian of rank $4$, contradicting Lemma~\ref{lem:two-commuting-hyperelliptic}.
\end{proof}







\begin{Proposition}\label{prop:final-one}
 Let  $S$ be a Sylow 2-subgroup of $G$ containing a hyperelliptic involution $h$ such that $Fix(h)$ has $r>2$ components and $B$ be  the subgroup of $S$ generated by the involutions with fixed-point set with $r$ components.  If $r$ is odd and $S$ contains at least two hyperelliptic involutions, then one of the two following conditions holds:
\begin{enumerate}
\item $G$ is solvable, $B\OO(G)$ is  normal in $G$ and  $G/B\OO(G)$ is isomorphic to a  quotient of a finite subgroup of $\SO4$;
\item $G/\OO(G)$ contains a unique component $Q/\OO(G)$ isomorphic  either to $\mathbb{A}_7$ or to   $\textrm{PSL}(2,q)$ with $q$ an odd prime power, and the  Fitting subgroup of $G/\OO(G)$ is cyclic or dihedral. The quotient $G/Q$ is a solvable group with cyclic Sylow $p$-subgroups for each $p$ odd.

\end{enumerate}

\end{Proposition}
\begin{proof}
Let $N$ be the subgroup normally generated by $B$ in $G$. We remark that $\OO(N)=\OO(G)\cap N.$ Cases (2a) or (2b) in Proposition~\ref{prop:component-normally-generated} hold.

First we suppose that case (2a) holds, that is $N=B \OO(N)$. In particular $B \OO(N)$ and $B \OO(G)$ are normal in $G$. By Remark~\ref{rem:normalizer-coprime-quotient} and isomorphisms Theorems, 
the group  $G/B\OO(G)$ is isomorphic to the  quotient of $\NN_G(B)$ by $B$;  by  Lemma~\ref{lem:quotient} the group $\NN_G(B)/B$ is isomorphic to a finite subgroup of $\SO4$. Since the automorphism group of a dihedral groups is solvable, the quotient   $\NN_G(B)/\CC_G(B)$ is solvable too. Moreover, the group  $\CC_G(B)$ commutes elementwise with the elementary abelian $2$-subgroup of rank two of $B$ generated by two hyperelliptic involutions.  By Lemma~\ref{lem:two-commuting-hyperelliptic}, the subgroup $\CC_G(B)$  is solvable.  Hence $\NN_G(B)$ is solvable and in this case we are done.

Suppose now to be in the case (2b) of Proposition~\ref{prop:component-normally-generated}. If $H$ is a subgroup of $G$, we denote by $\overline{H}$ the projection   $H\OO(G)/\OO(G)$ of $H$ to $G/\OO(G)$; in particular $\overline{G}$ is  $G/\OO(G).$
By the structure of $\overline{N}$ it follows that $\overline{G}/\CC_{\overline{G}}(\overline{N})$ has a unique simple section isomorphic either to $\mathbb{A}_7$ or to   $\textrm{PSL}(2,q)$. On the other hand, by Remark~\ref{rem:normalizer-coprime-quotient}, $\CC_{\overline{G}}(\overline{N})\leq \CC_{\overline{G}}(\overline{B})\cong \CC_G(B)/\CC_G(B)\cap \OO(G)$ which is solvable by Lemma~\ref{lem:two-commuting-hyperelliptic}. This implies that the layer of $\overline{G}$  is isomorphic either to $\mathbb{A}_7$ or to   $\textrm{PSL}(2,p^n)$.

Consider now  the Fitting subgroup $\mathcal{F}(\overline{G})$ of $\overline{G}$ which   is a 2-group by definition of $\OO(G).$  We remark that $\mathcal{F}(\overline{G})\cap \mathcal{E}(\overline{G})=\{1\}.$
The subgroup $\mathcal{F}(\overline{G})$ commutes with an elementary 2-subgroup of rank 2 generated by two projections of hyperelliptic involutions and contained in $\mathcal{E}(\overline{G}).$ By Remark~\ref{rem:normalizer-coprime-quotient} the subgroup $\mathcal{F}(\overline{G})$ is the projection of a  2-subgroup of $G$  commuting  with  an elementary 2-subgroup of rank 2 generated by two  hyperelliptic involutions. By  Lemma~\ref{lem:two-commuting-hyperelliptic} the subgroup $\mathcal{F}(\overline{G})$ is cyclic or dihedral. 
The product $\mathcal{F}(\overline{G})\times\mathcal{E}(\overline{G})$ is the generalized Fitting subgroup of $\overline{G}$, and hence the quotient  of $\overline{G}$ by the center of generalized Fitting subgroup  is isomorphic to a subgroup of $\textrm{Aut}( \mathcal{F}(\overline{G}))\times \textrm{Aut}(\mathcal{E}(\overline{G}))$. This implies that (2) holds.
 \end{proof}

We summarize  in the following theorem the results contained in Propositions~\ref{prop:component-normally-generated} and \ref{prop:final-one}.

\begin{Theorem}\label{final-theorem} 
Let  $S$ be a Sylow 2-subgroup of $G$ containing a hyperelliptic involution $h$ such that $Fix(h)$ has $r>2$ components. If $r$ is odd one of the following cases occurs:

\begin{enumerate}
\item the subgroup $\langle h,\OO(G) \rangle/\OO(G)$ lies in the centre of $G/\OO(G)$ and  $G/\langle h,\OO(G) \rangle$ is isomorphic to a  quotient of a finite subgroup of $\SO4 ;$

\item  $G$ is solvable and  there exists a dihedral 2-subgroup $B$ such that  $B\OO(G)$ is  normal in $G$ and  $G/B\OO(G)$ is isomorphic to a  quotient of a finite subgroup of $\SO4 ;$

\item $G/\OO(G)$ contains a normal subgroup $E\times F$ containing its centralizer; the subgroup $E$ is isomorphic  either to $\mathbb{A}_7$ or to   $\textrm{PSL}(2,q)$ with $q$ an odd prime number, and $F$ is cyclic or dihedral. The quotient of $(G/\OO(G))/ E\times F$ is a solvable group with cyclic Sylow p-subgroups for each $p$ odd.

\end{enumerate}

\end{Theorem}

\vskip1em

\subsubsection*{Acknowledgments}
I am  deeply indebted to Clara Franchi e Bruno Zimmermann for many valuable discussions on the topics of the paper. 

\vskip1em


\begin{thebibliography}{CCN{\etalchar{+}}85}

\bibitem[ABG70]{alperin-brauer-gorenstein}
J.~L. Alperin, Richard Brauer, and Daniel Gorenstein.
\newblock Finite groups with quasi-dihedral and wreathed {S}ylow
  {$2$}-subgroups.
\newblock {\em Trans. Amer. Math. Soc.}, 151:1--261, 1970.

\bibitem[BFMZ18]{boileau-et-al}
Michel Boileau, Clara Franchi, Mattia Mecchia, and Bruno Zimmermann.
\newblock Finite group actions on 3-manifolds and cyclic branched covers of
  knots.
\newblock {\em J.Topol.}, 11:283--308, 2018.

\bibitem[BLP05]{boileau-leeb-porti}
Michel Boileau, Bernhard Leeb, and Joan Porti.
\newblock Geometrization of 3-dimensional orbifolds.
\newblock {\em Ann. of Math. (2)}, 162(1):195--290, 2005.

\bibitem[BW71]{brauer-wong}
Richard Brauer and Warren~J. Wong.
\newblock Some properties of finite groups with wreathed {S}ylow
  {$2$}-subgroup.
\newblock {\em J. Algebra}, 19:263--273, 1971.

\bibitem[CCN{\etalchar{+}}85]{conway-et-al}
J.~H. Conway, R.~T. Curtis, S.~P. Norton, R.~A. Parker, and R.~A. Wilson.
\newblock {\em Atlas of finite groups}.
\newblock Oxford University Press, Eynsham, 1985.
\newblock Maximal subgroups and ordinary characters for simple groups, With
  computational assistance from J. G. Thackray.

\bibitem[CL00]{cooper-long}
D.~Cooper and D.~D. Long.
\newblock Free actions of finite groups on rational homology {$3$}-spheres.
\newblock {\em Topology Appl.}, 101(2):143--148, 2000.

\bibitem[CS03]{conway-smith}
John~H. Conway and Derek~A. Smith.
\newblock {\em On quaternions and octonions: their geometry, arithmetic, and
  symmetry}.
\newblock A K Peters, Ltd., Natick, MA, 2003.

\bibitem[Fuk80]{fukushima}
Hiroshi Fukushima.
\newblock Weakly closed dihedral {$2$}-subgroups in finite groups.
\newblock {\em J. Math. Soc. Japan}, 32(1):193--200, 1980.

\bibitem[GLS98]{gorenstein-lyons-solomon3}
Daniel Gorenstein, Richard Lyons, and Ronald Solomon.
\newblock {\em The classification of the finite simple groups. {N}umber 3.
  {P}art {I}. {C}hapter {A}}, volume~40 of {\em Mathematical Surveys and
  Monographs}.
\newblock American Mathematical Society, Providence, RI, 1998.
\newblock Almost simple $K$-groups.

\bibitem[Gol75]{goldschmidt}
David~M. Goldschmidt.
\newblock Strongly closed {$2$}-subgroups of finite groups.
\newblock {\em Ann. of Math. (2)}, 102(3):475--489, 1975.

\bibitem[GW65]{gorenstein-walter}
Daniel Gorenstein and John~H. Walter.
\newblock The characterization of finite groups with dihedral {S}ylow
  {$2$}-subgroups. {I}.
\newblock {\em J. Algebra}, 2:85--151, 1965.

\bibitem[Hal76]{hall}
J.~I. Hall.
\newblock Fusion and dihedral {$2$}-subgroups.
\newblock {\em J. Algebra}, 40(1):203--228, 1976.

\bibitem[McC02]{mccullough}
Darryl McCullough.
\newblock Isometries of elliptic 3-manifolds.
\newblock {\em J. London Math. Soc. (2)}, 65(1):167--182, 2002.

\bibitem[MS15]{mecchia-seppi}
Mattia Mecchia and Andrea Seppi.
\newblock Fibered spherical 3-orbifolds.
\newblock {\em Rev. Mat. Iberoam.}, 31(3):811--840, 2015.

\bibitem[MT07]{morgan-tian}
John Morgan and Gang Tian.
\newblock {\em Ricci flow and the {P}oincar\'e conjecture}, volume~3 of {\em
  Clay Mathematics Monographs}.
\newblock American Mathematical Society, Providence, RI; Clay Mathematics
  Institute, Cambridge, MA, 2007.

\bibitem[MZ04a]{mecchia-zimmermann2004-bis}
Mattia Mecchia and Bruno Zimmermann.
\newblock The number of knots and links with the same 2-fold branched covering.
\newblock {\em Q. J. Math.}, 55(1):69--76, 2004.

\bibitem[MZ04b]{mecchia-zimmermann2004}
Mattia Mecchia and Bruno Zimmermann.
\newblock On finite groups acting on {$Z_2$}-homology 3-spheres.
\newblock {\em Math. Z.}, 248(4):675--693, 2004.

\bibitem[RZ01]{reni-zimmermann2001}
Marco Reni and Bruno Zimmermann.
\newblock On hyperelliptic involutions of hyperbolic 3-manifolds.
\newblock {\em Math. Ann.}, 321(2):295--317, 2001.

\bibitem[Sak95]{sakuma-95}
Makoto Sakuma.
\newblock Homology of abelian coverings of links and spatial graphs.
\newblock {\em Canad. J. Math.}, 47(1):201--224, 1995.

\bibitem[Suz82]{suzuki1}
Michio Suzuki.
\newblock {\em Group theory. {I}}, volume 247 of {\em Grundlehren der
  Mathematischen Wissenschaften [Fundamental Principles of Mathematical
  Sciences]}.
\newblock Springer-Verlag, New York, 1982.

\bibitem[Suz86]{suzuki2}
Michio Suzuki.
\newblock {\em Group theory. {II}}, volume 248 of {\em Grundlehren der
  Mathematischen Wissenschaften [Fundamental Principles of Mathematical
  Sciences]}.
\newblock Springer-Verlag, New York, 1986.
\newblock Translated from the Japanese.

\bibitem[Zim15]{zimmermann2015}
Bruno~P. Zimmermann.
\newblock About three conjectures on finite group actions on 3-manifolds.
\newblock {\em Sib. \`Elektron. Mat. Izv.}, 12:955--959, 2015.

\end{thebibliography}

\newcommand{\etalchar}[1]{$^{#1}$}

\bigskip

\bigskip

\end{document}